\documentclass{article}

\usepackage[breaklinks,colorlinks,linkcolor=black,citecolor=black,urlcolor=black]{hyperref}
\usepackage{amsmath}
\usepackage{amsthm}
\usepackage{amsfonts}
\usepackage{mathrsfs}
\usepackage{enumerate}
\usepackage{setspace}
\usepackage{amssymb}
\usepackage{mathtools}
\usepackage{array}
\usepackage{graphicx}
\usepackage{float}
\usepackage{indentfirst}
\usepackage{bbm}
\usepackage[square, comma, sort&compress, numbers]{natbib}

\usepackage[noblocks]{authblk}

\newtheorem{theorem}{Theorem}[section]
\newtheorem{definition}[theorem]{Definition}
\newtheorem{lemma}[theorem]{Lemma}
\newtheorem{prop}[theorem]{Proposition}
\newtheorem{remark}[theorem]{Remark}
\newtheorem{hy}[theorem]{Assumption}

\allowdisplaybreaks[4]

\title{Forward-backward stochastic differential equations driven by G-Brownian motion}
\author[a]{Huan Lu}
\author[a]{Yongsheng Song \thanks{E-mail addresses: hlu@amss.ac.cn (H. Lu), yssong@amss.ac.cn (Y. Song).}}
\affil[a]{Academy of Mathematics and System Science, CAS, Beijing, China}
\date{ }

\begin{document}
	
	\maketitle
	
	\begin{abstract}
		In this paper, we study the existence and uniqueness of solutions to the fully coupled nonlinear forward-backward stochastic differential equations driven by G-Brownian motion in the following form, 
		\begin{equation*}
			\left\{
			\begin{array}{l}
				dX_t = b(t,X_t,Y_t) dt + h(t,X_t,Y_t) d\langle B \rangle_t + \sigma(t,X_t,Y_t) dB_t, \\
				dY_t = -f(t,X_t,Y_t,Z_t) dt - g(t,X_t,Y_t,Z_t) d \langle B \rangle_t + Z_t dB_t + dK_t, \\
				X_0 = x, \quad Y_T = \Phi(X_T), 
			\end{array}
			\right.
		\end{equation*}
		where $(B_t)_{t \geq 0}$ is a G-Brownian motion. Assuming that $\sigma$ is uniformly elliptic and coefficients are all differentiable, combining the results of fully nonlinear PDEs, we prove the existence and uniqueness of solutions to the equations above. 
	\end{abstract}
	
	
	\section{Introduction}
	
	In recent decades, forward-backward stochastic differential equations(FBSDEs, for short) in the Wiener probability space, which has a general form as 
	\begin{equation*}
		\left\{
		\begin{array}{l}
			dX_t = b(t,X_t,Y_t,Z_t) dt + \sigma(t,X_t,Y_t,Z_t)dW_t, \\
			dY_t = -f(t,X_t,Y_t,Z_t)dt + Z_tdW_t, \\
			X_0 = x, \quad Y_T = g(X_T), 
		\end{array}
		\right.
	\end{equation*}
	has been widely studied, where $(W_t)_{t \geq 0}$ denotes the standard Brownian motion. Many results, such as the well-posedness and regularity of their solutions, have been obtained by various means. There are three main methodologies to discuss the existence and uniqueness of the solutions to these FBSDEs: small time duration, the method of continuation, and the four-step scheme. \par 
	The fully coupled FBSDEs of the case $\sigma(t,x,y,z) = \sigma_0(t,x,y)$, was first researched by Antonelli(\cite{smallt}) in the early 1990s. By the Lipschitz property of the coefficients, Antonelli construct a contraction map when the time duration $T$ is small enough, which implies the unique solvability of FBSDEs. \par 
	The method of continuation was introduced in \cite{G-mon1, G-mon2, G-mon3} to discuss the solvability of fully coupled nonlinear FBSDEs. In these works, they introduced a so-called G-monotone conditions of coefficients, by which they can easily get the uniqueness of the solution if it exists. They then use this method to connect a family of FBSDEs, which have the same solvability, to get the existence of those solutions. \par 
	Unlike the two methods above, Ma, Protter and Yong (\cite{4step}) gave the solvability of FBSDEs based on the results of PDEs. Their approach is called the four-step scheme, which is extended to wider situations by Delarue (\cite{fbsdes}), where he supposed that $\sigma$ is independent of $z$ and is uniformly elliptic. More specifically, let $u$ be a solution to the PDE associated with the FBSDE, then we can solve the forward SDE by replacing $Y_t$ and $Z_t$ with $u(t,X_t)$ and $u_x(t,X_t)\sigma(t,X_t,u(t,X_t))$, respectively. In this way, $(X_t,Y_t,Z_t)$ is a solution to the FBSDE, where $Y_t = u(t,X_t)$ and $Z_t = u_x(t,X_t)\sigma(t,X_t,u(t,X_t))$. \par 
	Notice that the assumptions appearing in these three means cannot imply each other. For instance, method of continuation requires the G-monotone of coefficients, while Antonelli only needed the Lipschitz condition; both Antonelli and Delarue required $\sigma$ to be independent of $z$, but Delarue needed more  restrictions on $\sigma$ to get the solvability for arbitrary time duration $T$. Therefore, there is no unified approach to solve the problem. \par 
	We have already known that there is a close connection between PDEs and FBSDEs. Briefly speaking, the solution to a non-degenerate linear second-order parabolic partial differential equation can be expressed by a diffusion process, while the solution to quasi-linear ones can be represented by the solutions to FBSDEs. Nevertheless, in the linear expectation framework, there is no probabilistic representation for the solutions to fully nonlinear PDEs, because linear expectation cannot reflect the nonlinearity of these PDEs. Thanks to the G-expectation and G-Brownian motion introduced by Peng (\cite{G-expectation1, G-BM, G-expectation2, nonlinearexp}), the solutions to these PDEs have probabilistic representation with the help of FBSDEs in the G-expectation framework. In addition, FBSDEs can be obtained by applying stochastic maximum principle to optimal control problems. When the optimal control problems are considered in the G-expectation framework, which usually occurs in economics and finance, the associated FBSDEs should be driven by G-Brownian motion. \par 
	However, in the G-expectation framework, things go more complex. The solvability of SDEs and BSDEs are given by Peng (\cite{G-BM, nonlinearexp}), Bai and Lin (\cite{GSDE}), and Hu, Ji, Peng and Song (\cite{bsde}) respectively in recent years. To the best of our knowledge, there are few researchers studying FBSDEs under the G-expectation framework, named forward-backward stochastic differential equations driven by G-Brownian motion (FBGSDEs, for short), much less for those fully coupled ones. \par 
	Motivated by the aforementioned reasons, in this paper, we will discuss the solvability of FBGSDEs. Specifically, given a G-expectation space $(\Omega, \mathcal H, \hat{\mathbb E})$ and let $\{B_t\}_{t \geq 0}$ be a $d$-dimensional G-Brownian motion on it, we study whether there is a unique solution $(X_t, Y_t, Z_t, K_t)_{0 \leq t \leq T}$ to the following equation 	
	\begin{equation} \label{eq}
		\left\{
		\begin{array}{l}
			dX_t = b(t,X_t,Y_t) dt + h(t,X_t,Y_t) d\langle B \rangle_t + \sigma(t,X_t,Y_t) dB_t, \\
			dY_t = -f(t,X_t,Y_t,Z_t) dt - g(t,X_t,Y_t,Z_t) d \langle B \rangle_t + Z_t dB_t + dK_t, \\
			X_0 = x, \quad Y_T = \Phi(X_T), 
		\end{array}
		\right.
	\end{equation}
	where $X, Y, Z, K$ take values in $\mathbb R^n, \mathbb R^m, \mathbb R^{m \times d}, \mathbb R^m$, while $b,h,\sigma,f,g,\Phi$ take values in $\mathbb R^n, \mathbb R^n, \mathbb R^{n \times d}, \mathbb R^m, \mathbb R^m, \mathbb R^m$, respectively. For simplify, we only consider the case of $d=m=n=1$, and the multidimensional case is similar. The main difficulties in our work are the lack of dominated convergence theorem in the G-expectation framework and the asymmetry of G-martingale. In order to overcome these obstacles, in our paper, we use uniformly convergence and the general Doob's type inequality which needs higher moment estimations. \par 
	The rest of this paper is organized as follows. In section \ref{pre}, we shall present some preliminaries for the G-expectation framework and introduce the stochastic calculus on it. In section \ref{eusmall}, we give the solvability of FBGSDEs (\ref{eq}) on a small time duration, where we will also give a main estimation which plays a key role in our proof of the main theorem in this paper. In section \ref{nondeg}, some results of PDEs will be presented, and the existence and uniqueness of solutions to FBGSDEs (\ref{eq}) for arbitrary $T$, with non-degenerate diffusion, is provided, which is the main work of this paper. 
	
	\section{Preliminaries} \label{pre}
	
	In this section, we review some basic notations and results of sublinear expectation space and G-expectation, the readers may refer to \cite{G-expectation1, G-BM, G-expectation2, nonlinearexp, bsde, properties} for more details. 
	
	\subsection{Sublinear expectation space and G-expectation}
	
	Let $\Omega$ be a given set, and let $\mathcal H$ be a vector lattice of real valued functions defined on $\Omega$, which satisfies 
	\begin{enumerate}[(i)]
		\item $c \in \mathcal H$, for every $c \in \mathbb R$, 
		\item $|X| \in \mathcal H$, for any $X \in \mathcal H$. 
	\end{enumerate}
	In this article, the space $\mathcal H$ will be used as the space of random variables. 
	
	\begin{definition}
		A sublinear expectation is a function $\hat{\mathbb E} \colon \mathcal{H} \to \mathbb R$, satisfying 
		\begin{enumerate}[(i)]
			\item $\hat{\mathbb E} X \leq \hat{\mathbb E} Y$, for any $X \leq Y$, 
			\item $\hat{\mathbb E} c = c$, for any $c \in \mathbb R$, 
			\item $\hat{\mathbb E}[X+Y] \leq \hat{\mathbb E}X + \hat{\mathbb E}Y$, 
			\item $\hat{\mathbb E}[\lambda X] = \lambda \hat{\mathbb E}X$, for any $\lambda \geq 0$. 
		\end{enumerate}
		We then call the triple $(\Omega, \mathcal{H}, \hat{\mathbb E})$ a sublinear expectation space. 
	\end{definition}
	
	\begin{theorem}
		Let $\hat{\mathbb E}$ be a sublinear expectation on $\mathcal{H}$, then there exists a family of linear expectations $\{ \mathbb E_\theta \colon \theta \in \Theta \}$, such that 
		\begin{equation*}
			\hat{\mathbb E} X = \sup_{\theta \in \Theta} \mathbb E_\theta X. 
		\end{equation*}
	\end{theorem}

	\begin{definition}
		Let $(\Omega,\mathcal{H},\hat{\mathbb E})$ be a sublinear expectation space, a random vector $Y \in \mathcal{H}^m$ is said to be independent of another random vector $X \in \mathcal{H}^n$, if 
		\begin{equation*}
			\hat{\mathbb E}[\varphi(X,Y)] = \hat{\mathbb E}[ \hat{\mathbb E}[\varphi(x,Y)] \left|_{x = X}\right. ], \quad \text{for all } \varphi \in C_{l,lip}(\mathbb R^n \times \mathbb R^m). 
		\end{equation*}
	\end{definition}

	\begin{definition}
		Let $(\Omega,\mathcal{H},\hat{\mathbb E})$ be a sublinear expectation space, random vectors $X,Y \in \mathcal{H}^n$ are said to be identically distributed, denoted by $X \overset{d}{=} Y$, if
		\begin{equation*}
			\hat{\mathbb E}[\varphi(X)] = \hat{\mathbb E}[\varphi(Y)], \quad \text{for all } \varphi \in C_{l,lip}(\mathbb R^n). 
		\end{equation*}
	\end{definition}
	
	\begin{remark}
		We say $\tilde X$ is an independent copy of $X$, if $\tilde{X} \overset{d}{=} X$ and $\tilde{X}$ is independent of $X$. Notice that $\tilde{X}$ is independent of $X$ does not imply $X$ is independent of $\tilde X$.
	\end{remark}
	
	\begin{definition}[G-normal distribution]
		Given a sublinear expectation space $(\Omega,\mathcal{H},\hat{\mathbb E})$, random vector $X \in \mathcal{H}^d$ is called G-normally distributed, if 
		\begin{equation*}
			aX + b\tilde{X} \overset{d}{=} \sqrt{a^2+b^2} X, \quad \text{for any } a,b \geq 0, 
		\end{equation*}
		where $\tilde{X}$ is an independent copy of $X$. 
	\end{definition}
	
	Now we can define function 
	\begin{align*}
		G \colon S_d &\to \mathbb R, \\
		A &\mapsto \frac 12\hat{\mathbb E}[\langle AX, X \rangle], 
	\end{align*}
	where $S_d$ denotes the set of all $d \times d$ symmetric matrices. Since $G$ is a monotone, subadditive and bounded operator on $S_d$, there exists a bounded, convex and closed subset $\Gamma \subset S_d^+$, such that 
	\begin{equation} \label{G}
		G(A) = \frac 12\sup_{\gamma \in \Gamma} \rm{tr}(\gamma A). 
	\end{equation}
	Especially, we let $G$ be non-degenerate in this paper, i.e., there exits $\lambda > 0$, such that $G(A) - G(B) \geq \lambda \Vert A-B \Vert_{HS}$, for any $A \geq B$. 
	
	\begin{remark}
		In the case of $d=1$, $G(a) = \frac 12\overline \sigma^2 a^+ - \frac 12 \underline \sigma^2 a^-$, where $\overline \sigma^2 = \hat{\mathbb E}X^2$, $\underline \sigma^2 = -\hat{\mathbb E}[-X^2]$. And $G(a) - G(b) \geq \frac 12\underline{\sigma}^2(a-b)$, for any $a \geq b$. 
	\end{remark}
	
	\begin{definition}[G-expectation]
		Let $\Omega_T = C_0([0,T],\mathbb R^d)$ be the space of $\mathbb R^d$ valued continuous functions on $[0,T]$ with $\omega_0 = 0$, endowing with the supremum norm, and let $B_t(\omega) = \omega_t$ be the canonical process. Set $\mathcal H_T^0 = \{ \varphi(B_{t_1}, \cdots, B_{t_n}) \colon$ $t_1, \cdots, t_n \in [0,T], \varphi \in C_{l,lip}(\mathbb R^{d \times n}), n \geq 1 \}$, for any $X = \varphi(B_{t_1}-B_{t_0}, \cdots, B_{t_n}-B_{t_{n-1}}) \in \mathcal H_T^0$, we define G-expectation as follows, 
		\begin{align*}
			&\tilde{\mathbb E}[\varphi(B_{t_1}-B_{t_0}, B_{t_2}-B_{t_1}, \cdots, B_{t_n}-B_{t_{n-1}})] \\
			=& \hat{\mathbb E}[\varphi(\sqrt{t_1-t_0}\xi_1, \cdots, \sqrt{t_n-t_{n-1}}\xi_n)], 
		\end{align*}
		in which $\xi_1, \cdots, \xi_n$ are identically distributed d-dimensional G-normally distributed random vectors, and $\xi_{i+1}$ is independent of $(\xi_1, \cdots, \xi_i)$, $i = 1,2,\cdots,n-1$. We call $(\Omega_T, \mathcal H_T^0, \tilde{\mathbb E})$ the G-expectation space.  
	\end{definition}

	\begin{definition}[conditional expectation]
		Under the same notations above, for $X = \varphi(B_{t_1}-B_{t_0}, B_{t_2}-B_{t_1}, \cdots, B_{t_n}-B_{t_{n-1}})$, we can define its conditional expectation as follows, 
		\begin{align*}
			&\tilde{\mathbb E}_{t_i}[\varphi(B_{t_1}-B_{t_0}, B_{t_2}-B_{t_1}, \cdots, B_{t_n}-B_{t_{n-1}})] \\
			=& \tilde \varphi(B_{t_1}-B_{t_0}, B_{t_2}-B_{t_1}, \cdots, B_{t_i}-B_{t_{i-1}}), 
		\end{align*}
		where 
		\begin{equation*}
			\tilde \varphi(x_1,\cdots,x_i) = \tilde{\mathbb E}[\varphi(x_1,\cdots,x_i, B_{t_{i+1}}-B_{t_i}, \cdots, B_{t_n}-B_{t_{n-1}})]. 
		\end{equation*}
	\end{definition}
	
	\begin{remark}
		Without loss of generality, we write $\tilde{\mathbb E}$ as $\hat{\mathbb E}$ in the rest of the paper, and the canonical process $\{B_t\}_{t \in [0,T]}$ is the so-called G-Brownian motion. We refer readers to \cite{nonlinearexp} for more details of this part. 
	\end{remark}
	
	\subsection{Stochastic calculus in G-expectation}
	
	Similar with the classical stochastic analysis, here we can define integral of simple processes first, which have forms $\eta = \sum\limits_{i=0}^{n-1} \xi_i \mathbbm{1}_{[t_i,t_{i+1})}$, with respect to G-Brownian motion, then we extend it to some larger spaces. For readers' convenience, we list the main notations and spaces appearing in this paper as follows. 
	
	\begin{itemize}
		\item The scalar product and the norm of the Euclid space $\mathbb R^n$ are denoted by $\langle \cdot,\cdot \rangle$ and $|\cdot|$ respectively; 
		\item $L_{ip}(\Omega_T) = \{ \varphi(B_{t_1},\cdots,B_{t_n}) \colon t_i \in [0,T], \varphi \in C_{b,lip}({\mathbb R^{d \times n}}), n \geq 1 \}$; 
		\item $L_G^p(\Omega_T)$ is the completion of $L_{ip}(\Omega_T)$ under $\Vert \xi \Vert_{p,G} = (\hat{\mathbb E}|\xi|^p)^{1/p}$; 
		\item $M_G^0(0,T) = \{ \sum\limits_{i=0}^{n-1} \xi_i \mathbbm{1}_{[t_i,t_{i+1})} \colon 0 = t_0 < \cdots < t_n = T, \xi_i \in L_{ip}(\Omega_{t_i}) \}$; 
		\item $M_G^p(0,T)$ is the completion of $M_G^0(0,T)$ under $\Vert \eta \Vert_{M_G^p} = (\hat{\mathbb E} \int_0^T |\eta_t|^p dt)^{1/p}$; 
		\item $H_G^p(0,T)$ is the completion of $M_G^0(0,T)$ under $\Vert \eta \Vert_{H_G^p} = (\hat{\mathbb E} [\int_0^T |\eta_t|^2 dt]^{p/2})^{1/p}$; 
		\item $\mathbb L^p(\Omega_T) = \{ X \in \mathcal B(\Omega_T) \colon \hat{\mathbb E}|X|^p < \infty \}$; 
		\item $\mathbb M^{p,0}(0,T) = \{ \sum\limits_{i=0}^{n-1} \xi_i \mathbbm 1_{[t_i,t_{i+1})} \colon 0 = t_0 < \cdots < t_n = T, \xi_i \in \mathbb L^p(\Omega_{t_i}) \}$; 
		\item $\mathbb M^p(0,T)$ is the completion of $\mathbb M^{p,0}(0,T)$ under $\Vert \eta \Vert_{\mathbb M^p} = (\hat{\mathbb E}\int_0^T |\eta_t|^p dt)^{1/p}$; 
		\item $\mathbb H^p(0,T)$ is the completion of $\mathbb M^{p,0}(0,T)$ under $\Vert \eta \Vert_{\mathbb H^p} = (\hat{\mathbb E}[\int_0^T |\eta_t|^2 dt]^{p/2})^{1/p}$; 
		\item $S_G^0(0,T) = \{ \varphi(t,B_{t_1 \wedge t}, \cdots, B_{t_n \wedge t}) \colon t_1,\cdots,t_n \in [0,T], \varphi \in C_{b,lip}(\mathbb R^{n+1}) \}$; 
		\item $S_G^p(0,T)$ is the completion of $S_G^0(0,T)$ under $\Vert \eta \Vert_{S_G^p} = (\hat{\mathbb E}\sup\limits_{t \in [0,T]}|\eta_t|^p)^{1/p}$; 
		\item $\tilde S_G^p(0,T)$ is the completion of $S_G^0(0,T)$ under $\Vert \eta \Vert_{\tilde S_G^p} = (\sup\limits_{t \in [0,T]}\hat{\mathbb E}|\eta_t|^p)^{1/p}$. 
	\end{itemize}
	
	Similarly, in the G-expectation framework, we can define martingales and obtain some related properties, such as BDG inequality and Doob's maximal inequality, which play key roles in this paper. We first introduce the martingale property of stochastic integral with respect to G-Brownian motion and the BDG inequality, which can be found in \cite{nonlinearexp} and \cite{properties} for details. 
	
	\begin{prop} \label{bdg}
		For any $\eta \in \mathbb M^2(0,T)$, $p \geq 2$, there exist two constants $0 < c_p < C_P < \infty$ only depending on $p$, such that 
		\begin{align*}
			&\hat{\mathbb E} \int_0^T \eta_t dB_t = 0, \\
			\underline\sigma^p c_p \hat{\mathbb E}(\int_0^T |\eta_t|^2dt)^{p/2} \leq &\hat{\mathbb E}\sup\limits_{t \in [0,T]} |\int_0^t \eta_t dB_s|^p \leq \overline \sigma^p C_p \hat{\mathbb E}(\int_0^T |\eta_t|^2dt)^{p/2}, \\
			\hat{\mathbb E}\sup\limits_{t \in [0,T]} |\int_0^t \eta_t &d\langle B \rangle_s|^p \leq \overline \sigma^{2p} T^{p-1} \hat{\mathbb E}\int_0^T |\eta_t|^pdt. 
		\end{align*}
	\end{prop}
	
	The following Doob's type estimate is from Song (\cite{properties}), we also refer to the recently published book \cite{nonlinearexp} on nonlinear expectation by Shige Peng. 
	
	\begin{prop} \label{doob}
		For any $\alpha \geq 1$ and $\delta > 0$, let $1 < \gamma < \beta \coloneqq (\alpha + \delta)/\alpha$, $\gamma \leq 2$, there exists a positive constant $C_{\beta,\gamma}$,  only depending on $\beta$ and $\gamma$, such that for any $\xi \in L_{ip}(\Omega_T)$, the following inequality holds, 
		\begin{equation*}
			\hat{\mathbb E}\sup_{t \in [0,T]} \hat{\mathbb E}_t |\xi|^\alpha \leq C_{\beta,\gamma} [ (\hat{\mathbb E}|\xi|^{\alpha+\delta})^{1/\gamma\beta} + (\hat{\mathbb E}|\xi|^{\alpha+\delta})^{1/\gamma} ], 
		\end{equation*}
	\end{prop}
	
	\begin{remark} \label{redoob}
		Setting $C_\beta = 2 \inf \{ C_{\beta,\gamma} \colon 1 < \gamma < \beta, \gamma \leq 2 \}$, by Proposition \ref{doob}, and noticing $\frac{1}{\beta^2} < \frac{1}{\gamma\beta} < \frac{1}{\gamma} < 1$, we have 
		\begin{equation*}
			\hat{\mathbb E}\sup_{t \in [0,T]} \hat{\mathbb E}_t |\xi|^\alpha \leq C_{\beta} [ (\hat{\mathbb E}|\xi|^{\alpha+\delta})^{1/\beta^2} + \hat{\mathbb E}|\xi|^{\alpha+\delta} ]. 
		\end{equation*}
		Then for all $\lambda > 0$, 
		\begin{equation*}
			\hat{\mathbb E}\sup_{t \in [0,T]} \hat{\mathbb E}_t |\lambda\xi|^\alpha \leq C_{\beta} [ (\hat{\mathbb E}|\lambda\xi|^{\alpha+\delta})^{1/\beta^2} + \hat{\mathbb E}|\lambda\xi|^{\alpha+\delta} ], 
		\end{equation*}
		denoting $a = \hat{\mathbb E}\sup_{t \in [0,T]} \hat{\mathbb E}_t |\xi|^\alpha$ and $b = \hat{\mathbb E}|\xi|^{\alpha+\delta}$, we obtain 
		\begin{equation*}
			a\lambda^\alpha \leq C_\beta \big[ b^{1/\beta^2} \lambda^{(\alpha+\delta)/\beta^2} + b \lambda^{\alpha+\delta} \big], 
		\end{equation*}
		or equivalently, 
		\begin{equation*}
			a \leq C_\beta \big[ b^{1/\beta^2}\lambda^{-\delta/\beta} + b \lambda^{\delta} \big]. 
		\end{equation*}
		Thus, 
		\begin{equation*}
			a \leq \inf_{\lambda > 0} C_\beta \big[ b^{1/\beta^2}\lambda^{-\delta/\beta} + b \lambda^{\delta} \big] = C b^{1/\beta}, 
		\end{equation*}
		where $\lambda^* = b^{-1/(\alpha+\delta)} \beta^{-\beta/(1+\beta)\delta}$, which means 
		\begin{equation*}
			\hat{\mathbb E}\sup_{t \in [0,T]} \hat{\mathbb E}_t |\xi|^\alpha \leq C(\hat{\mathbb E}|\xi|^{\alpha+\delta})^{\alpha/(\alpha+\delta)}, 
		\end{equation*}
		where $C = \big( \beta^{1/(1+\beta)} + \beta^{-\beta/(1+\beta)} \big)C_\beta$. 
	\end{remark}
	
	The following lemma can be found in \cite{bsde} (Lemma 3.4), which seems simple but crucial. 
	
	\begin{lemma} \label{g-martin}
		Let $X \in S_G^2(0,T)$, $K^1, K^2$ are two non-increasing G-martingales with $K_0^1 = K_0^2 = 0$ and $K_T^1, K_T^2 \in L_G^2(\Omega_T)$, then 
		\begin{equation*}
			\int_0^t X_s^+ dK_s^1 + \int_0^t X_s^- dK_s^2 
		\end{equation*}
		is still a non-increasing G-martingale. 
	\end{lemma}
	
	\section{Existence and uniqueness in small time duration} \label{eusmall}
	
	\subsection{Existence and uniqueness}
	
	\begin{definition}
		We call the quadruple $(X_t, Y_t, Z_t, K_t)_{0 \leq t \leq T}$ a solution to the FBGSDE (\ref{eq}), if 
		\begin{enumerate}[(i)]
			\item $X, Y \in S_G^2(0,T)$, $Z \in M_G^2(0,T)$, and $K_t$ is a non-increasing G-martingale with $K_0 = 0$, $K_T \in L_G^2(\Omega_T)$; 
			\item $(X_t, Y_t, Z_t, K_t)_{0 \leq t \leq T}$ satisfies equation (\ref{eq}). 
		\end{enumerate}
	\end{definition}
	
	\begin{hy} \label{h1}
		We say real-valued functions $b, h, \sigma, f, g, \Phi$ satisfy Assumption \ref{h1}, if there exist two constants $L > 0$ and $\beta > 2$, such that 
		\begin{enumerate}[(i)]
			\item for every $t \in [0,T]$, and for every $(x,y,z), (x',y',z') \in \mathbb R^3$, 
			\begin{align*}
				|\varphi_1(t,x,y) - \varphi_1(t,x',y')| &\leq L (|x-x'| + |y-y'|), \\
				|\varphi_2(t,x,y,z) - \varphi_2(t,x',y',z')| &\leq L (|x-x'| + |y-y'| + |z-z'|), \\
				|\Phi(x) - \Phi(x')| &\leq L |x-x'|, 
			\end{align*} 
			where $\varphi_1 = (b, h, \sigma)^{\mathrm{T}}$, $\varphi_2 = (f, g)^{\mathrm{T}}$. 
			\item for every $t \in [0,T]$ and for every $(x,y,z) \in \mathbb R^3$, 
			\begin{align*}
				|\varphi_1(t,x,y)| &\leq L(1 + |y|), \\
				|\varphi_2(t,x,y,z) | &\leq L(1 + |y| + |z|). 
			\end{align*}
			\item For every $(x,y,z) \in \mathbb R^3$, we have $\varphi_1(\cdot,x,y), \varphi_2(\cdot,x,y,z) \in M_G^\beta(0,T)$, and $\Phi(x) \in L_G^\beta(\Omega_T)$. 
		\end{enumerate}
	\end{hy}
	
	\begin{theorem} \label{thm1}
		Assume that Assumption \ref{h1} holds, then there exists a constant $\delta = \delta(L)$, such that whenever $T \leq \delta$, FBGSDE (\ref{eq}) has a unique solution. 
	\end{theorem}
	
	\begin{proof}
		We first construct a map $\mathcal I \colon \tilde{S}_G^2(0,T) \to \tilde{S}_G^2(0,T)$ as follows, for any $(y_t)_{t \in [0,T]} \in \tilde{S}_G^2(0,T)$, let $(X_t)_{t \in [0,T]}$ be the solution of GSDE 
		\begin{equation} \label{fsde}
			X_t = x + \int_0^t b(s,X_s,y_s) ds + \int_0^t h(s,X_s,y_s) d\langle B \rangle_s + \int_0^t \sigma(s,X_s,y_s) dB_s, 
		\end{equation}
		and $(Y_t,Z_t,K_t)_{t \in [0,T]}$ be the solution of BGSDE 
		\begin{align} \label{bsde}
			Y_t = \Phi(X_T) &+ \int_t^T f(s,X_s,Y_s,Z_s) ds + \int_t^T g(s,X_s,Y_s,Z_s) d\langle B \rangle_s \notag\\
			&- \int_t^T Z_s dB_s - \int_t^T dK_s, 
		\end{align}
		then we define $\mathcal I(y) = Y$. Since the coefficients all satisfy Lipschitz conditions, the two equations above have unique solutions for any $T$ (see also \cite{nonlinearexp} and \cite{bsde}), the map $\mathcal{I}$ we constructed is well-defined. Next we show that, when $T$ is small enough, $\mathcal{I}$ is a contraction. \par
		For any $y, y' \in \tilde{S}_G^2(0,T)$, let $(X_t,Y_t,Z_t,K_t)_{t \in [0,T]}$ and $(X'_t,Y'_t,Z'_t,K'_t)_{t \in [0,T]}$ be the solutions of the corresponding equations. Set $\hat X_t = X_t - X'_t$, $\hat Y_t = Y_t - Y'_t$, $\hat y_t = y_t -y'_t$, noticing that 
		\begin{align*}
			\hat X_t =& \int_0^t b(s,X_s,y_s) - b(s,X'_s,y'_s) ds + \int_0^t h(s,X_s,y_s) - h(s,X'_s,y'_s) d\langle B \rangle_s \\
			&+ \int_0^t \sigma(s,X_s,y_s) - \sigma(s,X'_s,y'_s) dB_s, 
		\end{align*}
		by the Lipschitz conditions of $b$, $h$ and $\sigma$, and Proposition \ref{bdg}, we have 
		\begin{align*}
			\hat{\mathbb E} \sup_{s \in [0,t]} |\hat{X}_s|^2 &\leq CT \hat{\mathbb E} \int_0^t |\hat{X}_s|^2 + |\hat{y}_s|^2 ds + C \hat{\mathbb E} \int_0^t |\hat{X}_s|^2 + |\hat{y}_s|^2 ds \\
			&\leq C(T+1) \int_0^t \hat{\mathbb E} \sup_{r \in [0,s]}|\hat{X}_r|^2 ds + C(T^2 + T) \sup_{t \in [0,T]} \hat{\mathbb E}|\hat{y}_s|^2, 
		\end{align*}
		in which $C$ is a constant only depending on $L$, for convenience, here the constant $C$ can change from line to line. Using Gronwall inequality, we can obtain that 
		\begin{equation} \label{x1}
			\hat{\mathbb E} \sup_{t \in [0,T]} |\hat{X}_t|^2 \leq C_T \sup_{t \in [0,T]} \hat{\mathbb E}|\hat{y}_t|^2, 
		\end{equation}
		where $C_T$ is a constant depending on $T$ and $L$, with $C_T \to 0$, as $T \to 0$. \par
		Similarly, applying It\^o formula to $|\hat{Y}_t|^2$, we get 
		\begin{align*}
			|\hat{Y}_t|^2 =& |\Phi(X_T) - \Phi(X'_T)|^2 + 2\int_t^T \hat{Y}_s \cdot [f(s,X_s,Y_s,Z_s) - f(s,X'_s,Y'_s, Z'_s)] ds \\
			&+ 2\int_t^T \hat{Y}_s \cdot [g(s,X_s,Y_s,Z_s) - g(s,X'_s,Y'_s, Z'_s)] d\langle B \rangle_s \\
			&- 2\int_t^T \hat{Y}_s \cdot \hat{Z}_s dB_s -2\int_t^T \hat{Y}_s d\hat{K}_s - \int_t^T |\hat{Z}_s|^2 d\langle B \rangle_s, 
		\end{align*}
		where $\hat{Z}_t = Z_t - Z'_t$, $\hat{K}_t = K_t - K'_t$. By the Lipschitz conditions and Proposition \ref{bdg}, also noticing the inequality $2|\hat{Y}_t|\cdot|\hat{Z}_t| \leq \frac{1}{\varepsilon} |\hat{Y}_t|^2 + \varepsilon |\hat{Z}_t|^2$, we obtain 
		\begin{align*}
			|\hat{Y}_t|^2 \leq& L^2|\hat{X}_T|^2 + C\int_t^T |\hat{X}_s|^2 + |\hat{Y}_s|^2 ds - 2\int_t^T \hat{Y}_s \cdot \hat{Z}_s dB_s \\
			&+ 2(\int_t^T \hat{Y}_s^+ dK'_s + \int_t^T \hat{Y}_s^- dK_s) - 2(\int_t^T \hat{Y}_s^+ dK_s + \int_t^T \hat{Y}_s^- dK'_s) \\
			\leq& L^2|\hat{X}_T|^2 + C\int_t^T |\hat{X}_s|^2 + |\hat{Y}_s|^2 ds - 2(J_T - J_t), 
		\end{align*}
		in which $J_t \coloneqq \int_0^t \hat{Y}_s \cdot \hat{Z}_s dB_s + \int_0^t \hat{Y}_s^+ dK'_s + \int_t^T \hat{Y}_s^- dK_s$ is a G-martingale by Lemma \ref{g-martin}. Moving $(J_T-J_t)$ to the left hand side, taking conditional expectation $\hat{\mathbb E}_t$ first and then taking expectation $\hat{\mathbb E}$ on both sides, we get 
		\begin{align*}
			\hat{\mathbb E} |\hat{Y}_t|^2 \leq& L^2\hat{\mathbb E}|\hat{X}_T|^2 + C\hat{\mathbb E}\int_t^T |\hat{X}_s|^2 + |\hat{Y}_s|^2 ds \\
			\leq& L^2 \hat{\mathbb E}|\hat{X}_T|^2 + CT \hat{\mathbb E}\sup_{t \in [0,T]}|\hat{X}_t|^2 + C\int_t^T \hat{\mathbb E}|\hat{Y}_s|^2 ds, 
		\end{align*}
		by Gronwall inequality and $(\ref{x1})$, we have the following estimation 
		\begin{equation*}
			\hat{\mathbb E}|\hat{Y}_t|^2 \leq C_T \sup_{s \in [0,T]}\hat{\mathbb E}|\hat{y}_s|^2, 
		\end{equation*}
		thus $\sup\limits_{t \in [0,T]}\hat{\mathbb E}|\hat{Y}_t|^2 \leq C_T \sup\limits_{t \in [0,T]}\hat{\mathbb E}|\hat{y}_t|^2$, where $C_T \to 0$ as $T \to 0$. Then there exists a constant $\delta = \delta(L) > 0$, such that $C_T \leq \frac 12$ when $T \leq \delta$. In other words, $\mathcal I$ is contraction when $T$ is small enough. \par 
		By contraction mapping theorem, we know that there exists a unique process $Y \in \tilde{S}_G^2(0,T)$ satisfies equation (\ref{eq}). Noticing $S_G^2(0,T) \subset \tilde{S}_G^2(0,T)$, the linear growth and integrability of coefficients, we claim that $Y \in S_G^2(0,T)$ and the uniqueness still holds. Putting this $Y$ into GSDE (\ref{fsde}), we get a unique solution $X$; then putting $X$ into BGSDE (\ref{bsde}), by the existence and uniqueness of solutions to BGSDEs (see \cite{bsde}), we know the uniqueness of $(Z,K)$. Therefore, when $T$ is small enough, FBGSDE (\ref{eq}) has a unique solution $(X,Y,Z,K)$. 
	\end{proof}
	
	\subsection{Dependence upon coefficients}
	
	We have already got the existence and uniqueness of solutions to FBGSDEs which have forms (\ref{eq}), when $T$ is small enough. In this subsection, we will see how those solutions depend on their coefficients, which is the key estimate in this paper. 
	
	\begin{prop} \label{estimate}
		Assume that $(b,h,\sigma,f,g,\Phi)$ and $(b',h',\sigma',f',g',\Phi')$ satisfy Assumption \ref{h1}, and let $(X,Y,Z,K)$, $(X',Y',Z',K')$ be the solutions of FBGSDEs with corresponding coefficients, starting from $x$ and $x'$ respectively. Then there exist constants $\tilde{\delta} = \tilde{\delta}(L)$ and $C = C(L)$, such that for any $2 < 2+\alpha < \beta$, when $T \leq \tilde{\delta}$, we have 
		\begin{align*}
			&\hat{\mathbb E}\sup_{t \in [0,T]}|\hat{X}_t|^2 + \hat{\mathbb E}\sup_{t \in [0,T]}|\hat{Y}_t|^2 + \hat{\mathbb E}\int_0^T |\hat{Z}_t|^2 dt + \hat{\mathbb E} |\hat{K}_T|^2 \\
			\leq& C \big[ I_0 + I_\alpha + I_\alpha^{1/(2+\alpha)} \big], 
		\end{align*}
		in which 
		\begin{align*}
			I_\alpha =& |x-x'|^{2+\alpha} + \hat{\mathbb E} \int_0^T (|\hat b|^{2+\alpha} + |\hat h|^{2+\alpha} + |\hat \sigma|^{2+\alpha})(s,X_s,Y_s) ds \\
			&+ \hat{\mathbb E} |\hat \Phi(X_T)|^{2+\alpha} + \hat{\mathbb E} \int_0^T (|\hat f|^{2+\alpha} + |\hat g|^{2+\alpha})(s,X_s,Y_s,Z_s) ds, 
		\end{align*}
		\begin{equation*}
			\hat \varphi(t,x,y,z) = \varphi(t,x,y,z) - \varphi'(t,x,y,z), \quad \text{for} \quad \varphi = b,h,\sigma,f,g,\Phi. 
		\end{equation*}
	\end{prop}

	\begin{proof}
		Under the same notations used in the proof of Theorem \ref{thm1}. Noticing that $\hat{X}_t$ satisfies 
		\begin{align*}
			\hat{X}_t = \hat x &+ \int_0^t b'(s,X_s,Y_s) - b'(s,X'_s,Y'_s) ds + \int_0^t \hat{b}(s,X_s,Y_s) ds \\
			&+ \int_0^t h'(s,X_s,Y_s) - h'(s,X'_s,Y'_s) d\langle B \rangle_s + \int_0^t \hat{h}(s,X_s,Y_s) d\langle B \rangle_s \\
			&+ \int_0^t \sigma'(s,X_s,Y_s) - \sigma'(s,X'_s,Y'_s) dB_s + \int_0^t \hat{\sigma}(s,X_s,Y_s) dB_s, 
		\end{align*}
		by the Lipschitz conditions, Proposition \ref{bdg} and Gronwall inequality, with simple calculations, we have 
		\begin{equation} \label{x2}
			\hat{\mathbb E}\sup_{t \in [0,T]} |\hat{X}_t|^{2+\alpha} \leq C_T \sup_{t \in [0,T]}\hat{\mathbb E} |\hat{Y}_t|^{2+\alpha} + CI_\alpha. 
		\end{equation}
		Similarly, applying It\^o formula on $|\hat{Y}_t|^{2+\alpha}$, we get 
		\begin{align*}
			|\hat{Y}_t|^{2+\alpha} =& |\Phi'(X_T) - \Phi'(X'_T) + \hat{\Phi}(X_T)|^{2+\alpha} \\
			&+ (2+\alpha) \int_t^T |\hat{Y}_s|^{1+\alpha} \cdot [f'(s,X_s,Y_s,Z_s) - f'(s,X'_s,Y'_s,Z'_s)] ds \\
			&+ (2+\alpha) \int_t^T |\hat{Y}_s|^{1+\alpha} \cdot \hat{f}(s,X_s,Y_s,Z_s) ds \\
			&+ (2+\alpha) \int_t^T |\hat{Y}_s|^{1+\alpha} \cdot [g'(s,X_s,Y_s,Z_s) - g'(s,X'_s,Y'_s,Z'_s)] d\langle B \rangle_s \\
			&+ (2+\alpha) \int_t^T |\hat{Y}_s|^{1+\alpha} \cdot \hat{g}(s,X_s,Y_s,Z_s) d\langle B \rangle_s \\
			&- (2+\alpha) \int_t^T |\hat{Y}_s|^{1+\alpha} \cdot \hat{Z}_s dB_s - (2+\alpha) \int_t^T |\hat{Y}_s|^{1+\alpha} d\hat{K}_s \\
			&- \frac{(2+\alpha)(1+\alpha)}{2} \int_t^T |\hat{Y}_s|^\alpha \cdot |\hat{Z}_s|^2 d\langle B \rangle_s. 
		\end{align*}
		Using the same method in the proof of Theorem \ref{thm1}, combining (\ref{x2}), we find, when $T$ is small enough, 
		\begin{equation} \label{inq}
			\hat{\mathbb E}\sup_{t \in [0,T]}|\hat{X}_t|^{2+\alpha} + \sup_{t \in [0,T]}\hat{\mathbb E}|\hat{Y}_t|^{2+\alpha} \leq CI_\alpha. 
		\end{equation}
		Let $\alpha = 0$ we know that $\hat{\mathbb E}\sup_{t \in [0,T]}|\hat{X}_t|^2 \leq CI_0$. As for $\hat{\mathbb E}\sup_{t \in [0,T]}|\hat{Y}_t|^2$, noticing that for any $2 < 2+\alpha' < 2+\alpha < \beta$ (see Proposition 3.9 in \cite{bsde}), 
		\begin{align} \label{y1}
			\hat{\mathbb E}\sup_{t \in [0,T]}|\hat{Y}_t|^2 \leq& C \big[ \hat{\mathbb E}\sup_{t \in [0,T]}\hat{\mathbb E}_t|\hat{X}_T|^2 \\
			&+ \big( \hat{\mathbb E}\sup_{t \in [0,T]}\hat{\mathbb E}_t \left|\int_0^T f(s,X_s,Y_s,Z_s) - f'(s,X'_s,Y_s,Z_s) ds\right|^{2+\alpha'} \big)^{2/(2+\alpha')} \notag\\
			&+ \big( \hat{\mathbb E}\sup_{t \in [0,T]}\hat{\mathbb E}_t \left|\int_0^T g(s,X_s,Y_s,Z_s) - g'(s,X'_s,Y_s,Z_s) ds\right|^{2+\alpha'} \big)^{2/(2+\alpha')} \notag\\
			&+ \hat{\mathbb E}\sup_{t \in [0,T]}\hat{\mathbb E}_t \left|\int_0^T f(s,X_s,Y_s,Z_s) - f'(s,X'_s,Y_s,Z_s) ds\right|^{2+\alpha'} \notag\\
			&+ \hat{\mathbb E}\sup_{t \in [0,T]}\hat{\mathbb E}_t \left|\int_0^T g(s,X_s,Y_s,Z_s) - g'(s,X'_s,Y_s,Z_s) ds\right|^{2+\alpha'} \big], \notag
		\end{align}
		combining (\ref{inq}) and Proposition \ref{doob} with Remark \ref{redoob}, we obtain 
		\begin{equation*}
			\hat{\mathbb E}\sup_{t \in [0,T]}\hat{\mathbb E}_t|\hat{X}_T|^2 \leq C (\hat{\mathbb E}|\hat{X}_T|^{2+\alpha})^{2/(2+\alpha)} \leq C I_\alpha^{2/(2+\alpha)}, 
		\end{equation*}
		and 
		\begin{align*}
			&\hat{\mathbb E}\sup_{t \in [0,T]}\hat{\mathbb E}_t \left|\int_0^T f(s,X_s,Y_s,Z_s) - f'(s,X'_s,Y_s,Z_s) ds\right|^{2+\alpha'} \\
			\leq& C (\hat{\mathbb E} \left| \int_0^T f(s,X_s,Y_s,Z_s) - f'(s,X'_s,Y_s,Z_s) ds\right|^{2+\alpha})^{(2+\alpha')/(2+\alpha)},
		\end{align*}
		in which 
		\begin{align*}
			&\hat{\mathbb E} \left| \int_0^T f(s,X_s,Y_s,Z_s) - f'(s,X'_s,Y_s,Z_s) ds\right|^{2+\alpha} \\
			\leq& C \hat{\mathbb E} \int_0^T |\hat{X}_s|^{2+\alpha} + |\hat{f}|^{2+\alpha}(s,X_s,Y_s,Z_s) ds \leq CI_\alpha, 
		\end{align*}
		the terms involving $g$ and $g'$ is similar. Putting the inequalities above into (\ref{y1}), we get 
		\begin{equation*}
			\hat{\mathbb E}\sup_{t \in [0,T]}|\hat{Y}_t|^2 \leq C \big[ I_\alpha^{2/(2+\alpha)} + I_\alpha^{(2+\alpha')/(2+\alpha)} \big] \leq C\big[I_\alpha + I_\alpha^{2/(2+\alpha)} \big]. 
		\end{equation*}
		In the same way, by Proposition 3.8 in \cite{bsde}, with simple calculations, we have 
		\begin{equation*}
			\hat{\mathbb E}\int_0^T |\hat{Z}_t|^2 dt \leq C\big[I_\alpha + I_\alpha^{2/(2+\alpha)} + I_\alpha^{1/2} + I_\alpha^{1/(2+\alpha)}\big] \leq C\big[I_\alpha + I_\alpha^{1/(2+\alpha)}\big]. 
		\end{equation*}
		Finally, for $K$, noticing that 
		\begin{align*}
			\hat{K}_T =& \hat{Y}_T -\hat{Y}_0 + \int_0^T f(s,X_s,Y_s,Z_s) - f'(s,X'_s,Y'_s,Z'_s) ds \\
			&+ \int_0^T g(s,X_s,Y_s,Z_s) - g'(s,X'_s,Y'_s,Z'_s) d\langle B \rangle_s - \int_0^T \hat{Z}_s dB_s, 
		\end{align*}
		by simple calculations, we obtain 
		\begin{equation*}
			\hat{\mathbb E}|\hat{K}_T|^2 \leq C\big[ I_0 + I_\alpha + I_\alpha^{1/(2+\alpha)} \big]. 
		\end{equation*}
		In summary, we complete our proof. 
	\end{proof}
	
	\begin{remark}
		For convenience, when saying $\delta = \delta(L)$ in the rest of this paper we means $\delta = \delta(L) \wedge \tilde{\delta}(L)$, i.e., when $T \leq \delta$, FBGSDE (\ref{eq}) has a unique solution with the estimation above. 
	\end{remark}
	
	\begin{remark}
		$(b_n, h_n, \sigma_n, f_n, g_n, \Phi_n) \to (b, h, \sigma, f, g, \Phi)$ cannot imply the convergence of corresponding solutions or even $I_\alpha \to 0$, since the lack of dominated convergence theorem in the G-expectation framework. However, if we know that the coefficients are uniformly convergent, then their solutions will converge by Proposition \ref{estimate}. 
	\end{remark}
	
	\section{Non-degenerate diffusion coefficient case} \label{nondeg}
	
	\subsection{Fully nonlinear PDEs}
	
	At the beginning of this section, we introduce some results of fully nonlinear PDEs, which have forms 
	\begin{equation} \label{pde}
		\left\{
		\begin{array}{l}
			u_t + F(u_{ij}, u_i, u, t, x) = 0 \text{ in } Q, \\
			u = \varphi \text{ on } \partial'Q, 
		\end{array}
		\right.
	\end{equation}
	where $Q = (0,T) \times D$, $D \subset \mathbb R^n$, and $\partial'Q$ denotes the parabolic boundary of $Q$. We refer readers to \cite{pde} for more details. First of all, we introduce two important spaces $F(\varepsilon,K,Q)$ and $\overline F(\varepsilon,K,Q)$. 
	
	\begin{definition} \label{def-f}
		We say function $F(u_{ij}, u_i, u, t, x) \in F(\varepsilon,K,Q)$ if for every $t$, $F(u_{ij}, u_i, u, t, x)$ is twice continuous differentiable with respect to $(u_{ij},u_i,u,x)$, and for any symmetric matrix $(u_{ij})$, $F$ satisfies 
		\begin{enumerate}[(i)]
			\item $\varepsilon |\lambda|^2 \leq \sum F_{u_{ij}}\lambda_i\lambda_j \leq K |\lambda|^2$; 
			\item $|F - \sum F_{u_{ij}}u_{ij}| \leq M_1^F(u)(1 + \sum|u_i|^2)$; 
			\item $|F_{u_i}|(1+\sum|u_i|) + |F_u| + |F_x|(1+\sum|u_i|)^{-1} \leq M_1^F(u)(1+\sum|u_i|^2+\sum|u_{ij}|)$; 
			\item $[M_2^F(u,u_k)]^{-1} F_{(\eta)(\eta)} \leq \sum|\tilde u_{ij}| \big[ \sum|\tilde u_i| + (1+\sum|u_{ij}|)(|\tilde u| + |\tilde x|) \big]$ \\
			$\left.\right. $ \hspace{3.5cm} $+ \sum|\tilde u_i|^2(1+\sum|u_{ij}|) + (1+\sum|u_{ij}|^3)(|\tilde u|^2 + |\tilde x|^2)$; 
			\item $|F_t| \leq M_3^F(u,u_k)(1+\sum|u_{ij}|^2)$, 
		\end{enumerate}
		where $M_1^F(u), M_2^F(u,u_k), M_3^F(u,u_k)$ are some continuous functions which grow with $|u|$ and $|u_k|^2$, $M_2^F \geq 1$, $\eta = (\tilde u_{ij}, \tilde u_i, \tilde u, \tilde x)$, and 
		\begin{align*}
			F_{(\eta)(\eta)} \coloneqq& F_{u_{ij}u_{rs}} \tilde u_{ij} \tilde u_{rs} + 2F_{u_{ij}u_r} \tilde u_{ij}\tilde u_r + 2F_{u_{ij}u} \tilde u_{ij}\tilde u + 2F_{u_{ij}x} \tilde u_{ij}\tilde x \\
			&+ F_{u_iu_j} \tilde u_i\tilde u_j + 2F_{u_iu} \tilde u_i\tilde u + 2F_{u_ix} \tilde u_i\tilde x + F_{uu} \tilde u^2 + 2F_{ux} \tilde u\tilde x + F_{xx} \tilde x^2, 
		\end{align*}
	\end{definition}
	
	\begin{definition}
		We say $F(u_{ij},u_i,u,t,x) \in \overline F(\varepsilon,K,Q)$, if there exist $F_n \in F(\varepsilon,K,Q)$, such that $F_n \to F$, and satisfy 
		\begin{enumerate}[(i)]
			\item $M_i^{F_1} = M_i^{F_2} = \cdots \eqqcolon M_i^F$, $i = 1,2,3$; 
			\item $F_n$ is infinitely differentiable with respect to $(u_{ij},u_i,u,x)$; 
			\item there exist positive constants $\delta$ and $M_0$, such that for any $n \geq 1$ and for any $(u_{ij}) \in S_d^+$, 
			\begin{equation*}
				F_n(u_{ij},0,-M_0,t,x) \geq \delta_0, \quad F_n(-u_{ij},0,M_0,t,x) \leq -\delta_0, 
			\end{equation*}
		\end{enumerate}
	\end{definition}
	
	\begin{remark} \label{re}
		Let $\varepsilon \mathrm{I} \leq (A_{ij}) \leq K\mathrm{I}$, and let $A_{ij}, B_i, C$ be infinitely differentiable with bounded first two derivatives, which grow with $|u|$ and $|u_i|$. Constant $c \leq -L$, where $L$ denotes the growth speed of $C$ with respect to $u$, then it is easy to check that 
		\begin{equation*}
			F(u_{ij},u_i,u,t,x) = \sum_{ij}A_{ij}(t,x,u)u_{ij} + \sum_iB_i(t,x,u)u_i + cu + C(t,x,u,u_i)
		\end{equation*}
		belongs to $\overline F(\varepsilon, K, Q)$. 
	\end{remark}
	
	The next lemma is helpful for us to determine whether a function $F$ belongs to $\overline F(\varepsilon,K,Q)$, which can be found in Chapter 6, \cite{pde}. 
	
	\begin{lemma} \label{sup}
		Let $\{F_\lambda \colon \lambda \in \Lambda \}$ be a class of functions with $F_\lambda \in \overline F(\varepsilon,K,Q)$, if $M_i^{F_\lambda}, \delta_0^{F_\lambda}, M_0^{F_\lambda}$ are independent of $\lambda$, $i = 1,2,3$, then $F \coloneqq \sup\{ F_\lambda \colon \lambda \in \Lambda \} \in \overline F(\varepsilon,K,Q)$. 
	\end{lemma}
	
	Now we give the main theorem in this subsection, where we consider the PDE (\ref{pde}) on $Q = (0,T) \times \mathbb R^n$, which is a special case of Theorem 6.4.3 in \cite{pde}. 
	
	\begin{theorem} \label{pthm}
		Given $\varphi \in C(\mathbb R^n)$ with $|\varphi| \leq M_0^F$, then PDE (\ref{pde}) admits a solution $u \in C(\mathbb R^n)$ satisfying $|u| \leq M_0^F$. Moreover, for any $0 < k < 1$, $u \in C_b^{2+\alpha_0}((0,T-k^2) \times \mathbb R^n)$, where $\alpha_0 = \alpha_0(n,\varepsilon,K)$, and its norm in this space is bounded by a constant $C(n,\varepsilon,K,M_i^F,M_0^F,k)$. 
	\end{theorem}
	
	\subsection{Arbitrary time duration $T$}
	
	With all the preparations above, we can discuss the existence and uniqueness of solutions to the FBGSDEs with forms (\ref{eq}) for any $T > 0$. We first list the basic assumptions and give the main result of this paper. 
	
	\begin{hy} \label{h2}
		Assume $(b,h,\sigma,f,g,\Phi)$ are continuous differentiable with respect to $(x,y,z)$, we say they satisfy Assumption \ref{h2} if there exist constants $L, M, \lambda > 0$, $\beta > 2$, such that they satisfy both Assumption \ref{h1} with constants $L$ and $\beta$, and the following properties, 
		\begin{enumerate}[(i)]
			\item For any $t,s \in [0,T]$ and for any $(x,y,z) \in \mathbb R^3$, 
			\begin{align*}
				|\varphi_1(t,x,y) - \varphi_1(s,x,y)| \leq& L |t-s|, \\
				|\varphi_2(t,x,y,z) - \varphi_2(s,x,y,z)| \leq& L |t-s|,
			\end{align*} 
			where $\varphi_1 = (b,h,\sigma)^{\mathrm{T}}$, $\varphi_2 = (f,g)^{\mathrm{T}}$. 
			\item For any $t \in [0,T]$ and for any $(x,y,z), (x',y',z') \in \mathbb R^3$, 
			\begin{align*}
				|\partial_x \varphi_1(t,x,y) - \partial_x \varphi_1(t,x',y')| \leq& L(|x-x'| + |y-y'|), \\
				|\partial_y \varphi_1(t,x,y) - \partial_y \varphi_1(t,x',y')| \leq& L(|x-x'| + |y-y'|), \\
				|\partial_z \varphi_1(t,x,y) - \partial_z \varphi_1(t,x',y')| \leq& L(|x-x'| + |y-y'|), \\
				|\partial_x \varphi_2(t,x,y,z) - \partial_x \varphi_2(t,x',y',z')| \leq& L(|x-x'| + |y-y'| + |z-z'|), \\
				|\partial_y \varphi_2(t,x,y,z) - \partial_y \varphi_2(t,x',y',z')| \leq& L(|x-x'| + |y-y'| + |z-z'|), \\
				|\partial_z \varphi_2(t,x,y,z) - \partial_z \varphi_2(t,x',y',z')| \leq& L(|x-x'| + |y-y'| + |z-z'|), \\
				|\partial_x \Phi(x) - \partial_x \Phi(x')| \leq& L(|x-x'|). 
			\end{align*}
			\item For any $t \in [0,T]$ and for any $(x,y) \in \mathbb R$, 
			\begin{align*}
				|\sigma(t,x,y)| + |\Phi(x)| \leq& M, \\
				\sigma^2(t,x,y) \geq& \lambda. 
			\end{align*}
		\end{enumerate}
	\end{hy}
	
	\begin{theorem} \label{thm2}
		For any $T > 0$, if Assumption \ref{h2} holds, then FBGSDE (\ref{eq}) admits a unique solution. 
	\end{theorem}
	
	Before proving Theorem \ref{thm2}, we give the Feynman-Kac formula that the solution of FBGSDE (\ref{eq}) should satisfy, where we assume that all the coefficients are smooth enough. We can define function $u \colon [0,T] \times \mathbb R \to \mathbb R$ as $u(t,x) = Y_t^{t,x}$, when $T$ is small enough, where $(X^{t,x},Y^{t,x},Z^{t,x},K^{t,x})$ is the unique solution of 
	\begin{equation} \label{eqtx}
		\left\{
		\begin{array}{l}
			dX_s^{t,x} = b(s,X_s^{t,x},Y_s^{t,x}) ds + h(s,X_s^{t,x},Y_s^{t,x}) d\langle B \rangle_s + \sigma(s,X_s^{t,x},Y_s^{t,x}) dB_s, \\
			dY_s^{t,x} = -f(s,X_s^{t,x},Y_s^{t,x},Z_s^{t,x}) ds - g(s,X_s^{t,x},Y_s^{t,x},Z_s^{t,x}) d\langle B \rangle_s + Z_s^{t,x} dB_s + dK_s^{t,x}, \\
			X_t^{t,x} = x, \quad Y_T^{t,x} = \Phi(X_T^{t,x}). 
		\end{array}
		\right. 
	\end{equation}
	Suppose that $u \in C_b^{1,2}$ on which we can apply It\^o formula. By the uniqueness of solutions we know that $Y_t = u(t,X_t)$, then applying It\^o formula on $u(t,X_t)$, we can find that  $Z_t = u_x(t,X_t) \cdot \sigma(t,X_t,u(t,X_t))$, $K_t = \int_0^t \frac 12u_{xx}\sigma^2 + u_xh + g d\langle B \rangle_s - \int_0^t G(u_{xx}\sigma^2 + 2u_xh + 2g) ds$, and $u$ satisfies 
	\begin{equation} \label{pde*}
		\left\{
		\begin{array}{l}
			u_t(t,x) + u_x(t,x)b(t,x,u(t,x)) + G(u_{xx}(t,x)\sigma^2(t,x,u(t,x)) \\
			+ 2u_x(t,x)h(t,x,u(t,x)) + 2g(t,x,u(t,x),u_x(t,x)\sigma(t,x,u(t,x)))) \\
			+ f(t,x,u(t,x),u_x(t,x)\sigma(t,x,u(t,x))) = 0, \\
			u(T,x) = \Phi(x). 
		\end{array}
		\right.
	\end{equation}
	
	\begin{lemma} \label{solution}
		For any $T > 0$, if $(b,h,\sigma,f,g,\Phi)$ are infinitely differentiable with bounded first two derivatives, whose boundaries are denoted by $L$, $(\sigma, \Phi)$ are bounded, and $\sigma$ satisfies uniformly elliptic condition. Also let Assumption \ref{h1} hold true, then there is a solution to PDE (\ref{pde*}). 
	\end{lemma}
	\begin{proof}
		If $u$ is a solution of (\ref{pde*}), set $\tilde{u}(t,x) = e^{-L(T-t)}u(t,x)$, then $\tilde{u}$ satisfies 
		\begin{equation*}
			\left\{
			\begin{array}{l}
				\tilde{u}_t + \tilde{u}_x\tilde b + G(\tilde u_{xx}\tilde\sigma^2+2\tilde u_x\tilde h+2\tilde g) + \tilde f -L\tilde u = 0, \\
				\tilde u(T,x) = \Phi(x), 
			\end{array}
			\right.
		\end{equation*}
		where $\tilde b = b(t,x,e^{L(T-t)}\tilde u(t,x))$, and $\tilde h, \tilde \sigma, \tilde f, \tilde g$ are similar. Noticing that $G(\cdot)$ can be expressed as $\frac 12\sup_{\gamma \in \Gamma} \rm{tr}(\gamma \cdot)$ (see (\ref{G})), combining with remark $\ref{re}$ and lemma \ref{sup}, we know that $\tilde{u}_x\tilde b + G(\tilde u_{xx}\tilde\sigma^2+2\tilde u_x\tilde h+2\tilde g) + \tilde f -L\tilde u \in \overline{F}(\varepsilon,K,Q)$. Then by Theorem \ref{pthm}, the PDE above admits a solution $\tilde u$. Thus, $u(t,x) = e^{L(T-t)}\tilde u(t,x)$ is a solution of (\ref{pde*}). 
	\end{proof}
	
	\begin{prop} \label{smoothexist}
		Under the assumptions in Lemma \ref{solution}, for any $T > 0$, there exists a solution to the FBGSDE (\ref{eq}). 
	\end{prop}
	
	\begin{proof}
		By Lemma \ref{solution}, there exists a solution to the PDE (\ref{pde*}), which we denote by u. Then there exists a constant $M_0 > 0$, such that $|u| \leq M_0$; and for any $\delta > 0$, there exists a constant $M$, such that $\Vert u \Vert_{C^2((0,T-\delta) \times \mathbb R)} \leq M$. Then the following GSDE 
		\begin{align*}
			X_t =& x + \int_0^t b(s,X_s,u(s,X_s))ds + \int_0^t h(s,X_s,u(s,X_s))d \langle B \rangle_s \\
			&+ \int_0^t \sigma(s,X_s,u(s,X_s))dB_s, 
		\end{align*}
		admits a unique solution $\{X_t\}_{t \in [0,T)}$. Recalling the linear growth of $b,h,\sigma$ with respect to $y$, and the boundedness of $u$, we have the following estimates, 
		\begin{align*}
			\hat{\mathbb E}\sup_{s \in [0,t]}|X_s|^2 \leq& C|x|^2 + C \hat{\mathbb E} \int_0^t (1 + |u(s,X_s)|^2) ds \\
			\leq& C + C \hat{\mathbb E} \int_0^T (1+M_0^2) ds \\
			\leq& M_T, 
		\end{align*}
		where $M_T$ is a constant only depending on $T$, and for any $0 \leq s \leq t < T$, 
		\begin{equation*}
			\hat{\mathbb E} |X_s - X_t|^2 \leq C\int_s^t (1 + M_0^2) dr \leq C_T|t-s|, 
		\end{equation*}
		which implies $\{X_t\}_{t \in [0,T)}$ is a Cauchy sequence in $L_G^2(\Omega_T)$, then $X_T \coloneqq \lim\limits_{t \to T} X_t$ does exist. Moreover, $\{X_t\}_{t \in [0,T]}$ satisfies the GSDE above. Then let $Y_t = u(t,X_t)$, $Z_t = u_x(t,X_t) \cdot \sigma(t,X_t,u(t,X_t))$, and $K_t = \int_0^t \frac 12u_{xx}\sigma^2 + u_xh + g d\langle B \rangle_s - \int_0^t G(u_{xx}\sigma^2 + 2u_xh + 2g) ds$, it is not hard to verify that $(X,Y,Z,K)$ is a solution to the FBGSDE (\ref{eq}). 
	\end{proof}
	
	\begin{proof}[Proof of Theorem \ref{thm2}]
		Let $\rho(x) = c_0e^{-1/(1-x^2)}\mathbbm 1_{(-1,1)}(x)$, where $c_0 > 0$ is a constant such that $\int \rho = 1$, then $\{\rho_n(x) = n\rho(nx)\}$ is a family of kernels, and let $(b_n,h_n,\sigma_n,f_n,g_n,\Phi_n)$ be the convolutions of $(b,h,\sigma,f,g,\Phi)$ with $\rho_n$, for $n \geq 1$. Then  by Assumption \ref{h2}, $(b_n,h_n,\sigma_n,f_n,g_n,\Phi_n) \to (b,h,\sigma,f,g,\Phi)$ uniformly, which also satisfy the assumptions in Lemma \ref{solution}. Hence there exist $u_n \in C^{2+\alpha_0}$, which are solutions to PDEs with coefficients $(b_n,h_n,\sigma_n,f_n,g_n,\Phi_n)$, satisfying that for any $k \in (0,1)$, there exists $M > 0$, such that 
		\begin{equation} \label{k}
			\Vert u_n \Vert_{C^{2+\alpha_0}((0,T-k^2) \times \mathbb R)} \leq M. 
		\end{equation}
		And let $\Theta^{n,t,x} = (X^{n,t,x},Y^{n,t,x},Z^{n,t,x},K^{n,t,x})$ be the solutions of FBGSDEs (\ref{eqtx}) with coefficients $(b_n,h_n,\sigma_n,f_n,g_n,\Phi_n)$, which satisfy $Y_s^{n,t,x} = u_n(s,X_s^{n,t,x})$. \par 
		Step 1. We first construct a function $u \colon [0,T] \times \mathbb R \to \mathbb R$ by solutions of FBGSDEs, which is also the limit of $u_n$. \par 
		By Theorem \ref{thm1}, there exists a constant $\delta_1 = \delta(L)$, such that the FBGSDEs (\ref{eqtx}) admit unique solutions $\Theta^{t,x} = (X^{t,x},Y^{t,x},Z^{t,x},K^{t,x})$ for $T-t \leq \delta_1$, and we can define function $u^0(t,x)$ on $[T-\delta_1,T] \times \mathbb R$ as $u^0(t,x) = Y_t^{t,x}$, which satisfies $Y_t^{T-\delta_1,x} = u^0(t,X_t^{T-\delta_1,x})$. Then by Proposition $\ref{estimate}$, we have $(\Theta_t^{n,T-\delta_1,x})_{t \in [T-\delta_1,T]} \to (\Theta_t^{T-\delta_1,x})_{t \in [T-\delta_1,T]}$, as $n \to \infty$. Moreover, by the uniformly convergence of coefficients, we have $u_n\left|_{[T-\delta_1,T]}\right. \to u^0$ uniformly. Specially, $u_n(T-\delta_1,\cdot) \to u(T-\delta_1,\cdot)$ uniformly. Taking $k^2 = \delta_1/2$ in (\ref{k}), we have 
		\begin{equation*}
			|u_n(t,x) - u_n(t,y)| \leq M|x-y|, \quad \text{for any } t \in (0,T-\frac{\delta_1}{2}), x \in \mathbb R, n \geq 1, 
		\end{equation*}
		so $|u^0(T-\delta_1,x) - u^0(T-\delta_1,y)| \leq M|x-y|$. \par 
		Set $\delta_0 = \delta(M \vee L)$, and divide $(0,T-\delta_1)$ as $0 = t_0 < t_1 < \cdots < t_N = T-\delta_1$, in which $t_{i+1} - t_i \leq \delta_0$, $i = 0,1,\cdots,N-1$, denoting $t_{N+1} = T$ for convenience. We consider the following FBGSDE: 
		\begin{equation*}
			\left\{
			\begin{array}{l}
				dX_t = b(t,X_t,Y_t) dt + h(t,X_t,Y_t) d\langle B \rangle_t + \sigma(t,X_t,Y_t) dB_t, \\
				dY_t = -f(t,X_t,Y_t,Z_t) dt - g(t,X_t,Y_t,Z_t) d \langle B \rangle_t + Z_t dB_t + dK_t, \\
				X_{t_{N-1}} = x, \quad Y_{t_N} = u^0(t_N, X_{t_N}). 
			\end{array}
			\right.
		\end{equation*}
		By Theorem \ref{thm1}, the equation above admits a unique solution, and we can define function $u^1 \colon [t_{N-1},t_N] \times \mathbb R \to \mathbb R$ as before, which also satisfies $u^1(t_N,\cdot) = u^0(t_N,\cdot)$. From Proposition \ref{estimate}, we have $u_n\left|_{[t_{N-1},t_N]}\right. \to u^1$, which implies $|u^1(t,x) - u^1(t,y)| \leq M|x-y|$, for any $t \in [t_{N-1},t_N]$, and for any $x, y \in \mathbb R$. Also we have $u_n(t_{N-1},\cdot) \to u(t_{N-1},\cdot)$ uniformly. \par 
		In the same way, consider the FBGSDE 
		\begin{equation*}
			\left\{
			\begin{array}{l}
				dX_t = b(t,X_t,Y_t) dt + h(t,X_t,Y_t) d\langle B \rangle_t + \sigma(t,X_t,Y_t) dB_t, \\
				dY_t = -f(t,X_t,Y_t,Z_t) dt - g(t,X_t,Y_t,Z_t) d \langle B \rangle_t + Z_t dB_t + dK_t, \\
				X_{t_{N-2}} = x, \quad Y_{t_{N-1}} = u^1(t_{N-1}, X_{t_{N-1}}). 
			\end{array}
			\right.
		\end{equation*}
		By Theorem \ref{thm1} and Proposition \ref{estimate}, we can define $u^2 \colon [t_{N-2},t_{N-1}] \times \mathbb R \to \mathbb R$, $u^2(t_{N-1},\cdot) = u^1(t_{N-1},\cdot)$, with $|u^2(t,x) - u^2(t,y)| \leq M|x-y|$, for any $t \in [t_{N-2},t_{N-1}]$ and for any $x,y \in \mathbb R$. Moreover, $u_n(t_{N-2},\cdot) \to u(t_{N-2},\cdot)$ uniformly. \par 
		Repeat the procedure finite times, we are able to define a function on $[0,T] \times \mathbb R$ as follow, 
		\begin{equation*}
			u(t,x) \coloneqq \sum_{i=1}^{N} u^i(t,x) \mathbbm 1_{[t_{N-i},t_{N-i+1})}(t) + u^0(t,x) \mathbbm 1_{[t_N,T]}(t), 
		\end{equation*}
		which satisfies $Y_t^{t_i,x} = u(t,X_t^{t_i,x})$, for $t \in [t_i,t_{i+1}]$, $i = 0,\cdots,N$, and $u_n(t,\cdot) \to u(t,\cdot)$ uniformly. \par 
		Step 2. We construct process $(X^*,Y^*) \in S_G^2(0,T) \times S_G^2(0,T)$ with $Y_t = u(t,X_t)$ in this step. \par 
		Now we can solve the following FBGSDE, 
		\begin{equation*}
			\left\{
			\begin{array}{l}
				dX_t = b(t,X_t,Y_t) dt + h(t,X_t,Y_t) d\langle B \rangle_t + \sigma(t,X_t,Y_t) dB_t, \\
				dY_t = -f(t,X_t,Y_t,Z_t) dt - g(t,X_t,Y_t,Z_t) d \langle B \rangle_t + Z_t dB_t + dK_t, \\
				X_{t_0} = 0, \quad Y_{t_1} = u(t_1, X_{t_1}), 
			\end{array}
			\right.
		\end{equation*}
		whose solution we denote by $\{(X_t^{(0)},Y_t^{(0)},Z_t^{(0)},K_t^{(0)})\}_{t \in [t_0,t_1]}$. Then we can define $(X_t^*,Y_t^*) = (X_t^{(0)}, Y_t^{(0)})$ for $t \in [t_0,t_1]$. \par 
		Then the FBGSDE 
		\begin{equation*}
			\left\{
			\begin{array}{l}
				dX_t = b(t,X_t,Y_t) dt + h(t,X_t,Y_t) d\langle B \rangle_t + \sigma(t,X_t,Y_t) dB_t, \\
				dY_t = -f(t,X_t,Y_t,Z_t) dt - g(t,X_t,Y_t,Z_t) d \langle B \rangle_t + Z_t dB_t + dK_t, \\
				X_{t_1} = X_{t_1}^*, \quad Y_{t_2} = u(t_2, X_{t_2}), 
			\end{array}
			\right.
		\end{equation*}
		admits a unique solution $(X^{(1)},Y^{(1)},Z^{(1)},K^{(1)})$. Since $Y_t^{n,0,x} = u_n(t,X_t^{n,0,x})$, $u_n(t,\cdot) \to u(t,\cdot)$ uniformly, and $(X^{n,0,x},Y^{n,0,x}) \to (X^{(1)},Y^{(1)})$ in $S_G^2(t_1,t_2)$, which follows from Proposition \ref{estimate}, we have $Y_{t_1}^{(1)} = u(t_1,X_{t_1}^{(1)}) = Y_{t_1}^*$. Thus we can define $(X_t^*,Y_t^*) = (X_t^{(1)}, Y_t^{(1)})$ for $t \in [t_1,t_2]$, which is well-defined. \par 
		Similarly, let $(X^{(2)},Y^{(2)},Z^{(2)},K^{(2)})$ be the solution of the following FBGSDE, 
		\begin{equation*}
			\left\{
			\begin{array}{l}
				dX_t = b(t,X_t,Y_t) dt + h(t,X_t,Y_t) d\langle B \rangle_t + \sigma(t,X_t,Y_t) dB_t, \\
				dY_t = -f(t,X_t,Y_t,Z_t) dt - g(t,X_t,Y_t,Z_t) d \langle B \rangle_t + Z_t dB_t + dK_t, \\
				X_{t_2} = X_{t_2}^*, \quad Y_{t_3} = u(t_3, X_{t_3}), 
			\end{array}
			\right.
		\end{equation*}
		which satisfies $Y_t^{(2)} = u(t,X_t^{(2)})$. Hence we can define $(X_t^*,Y_t^*) = (X_t^{(2)}, Y_t^{(2)})$ for $t \in [t_2,t_3]$. \par 
		Repeat the procedure finite times, we obtain processes $X^*$ and $Y^*$ on $[0,T]$, which satisfies $Y_t^* = u(t,X_t^*)$. \par 
		Step 3. We construct $Z^*$ and $K^*$, then we show that $(X^*,Y^*,Z^*,K^*)$ is a solution to the FBGSDE (\ref{eq}). \par 
		It is obvious that $X^*$ is a continuous process satisfying the equation
		\begin{equation} \label{gsde}
			\left\{
			\begin{array}{l}
				dX_t = b(t,X_t,Y_t^*)dt + h(t,X_t,Y_t^*)d \langle B \rangle_t + \sigma(t,X_t,Y_t^*)dB_t, \\
				X_0 = x,  
			\end{array}
			\right. 
		\end{equation}
		and that 
		\begin{equation*}
			\Vert X^* \Vert_{S_G^2(0,T)} \leq \sum_{i=0}^{N} \Vert X^* \Vert_{S_G^2(t_i,t_{i+1})} < \infty, 
		\end{equation*}
		which means $X^*$ is a solution to the above GSDE. By the solvability of BGSDE (see \cite{bsde}), the following equation admits a unique solution $(\tilde{Y},\tilde{Z},\tilde{K})$, 
		\begin{equation} \label{bgsde}
			\left\{
			\begin{array}{l}
				dY_t = -f(t,X_t^*,Y_t,Z_t)dt - g(t,X_t^*,Y_t,Z_t)d \langle B \rangle_t + Z_tdB_t + dK_t, \\
				Y_T = \Phi(X_T^*). 
			\end{array}
			\right.
		\end{equation}
		Since $\{(Y_t^*,Z_t^{(N)},K_t^{(N)})\}_{t \in [t_N,t_{N+1}]}$ solve the same BGSDE above on $[t_N,t_{N+1}]$, we have $Y^* = \tilde{Y}$ on $[t_N,t_{N+1}]$. By the argument before, we obtain that $\tilde{Y}_{t_N} = Y^*_{t_N} = u(t_N,X_{t_N}^*)$, which implies that $\{(\tilde{Y}_t,\tilde{Z}_t,\tilde{K}_t)\}_{t \in [0,t_N]}$ is the unique solution of the following GBSDE 
		\begin{equation*}
			\left\{
			\begin{array}{l}
				dY_t = -f(t,X_t^*,Y_t,Z_t)dt - g(t,X_t^*,Y_t,Z_t)d \langle B \rangle_t + Z_tdB_t + dK_t, \\
				Y_{t_N} = u(t_N,X_{t_N}^*). 
			\end{array}
			\right.
		\end{equation*}
		Similarly, since $\{(Y_t^*,Z_t^{(N-1)},K_t^{(N-1)})\}_{t \in [t_{N-1},t_N]}$ solves the equation above, we obtain $Y^* = \tilde{Y}$ on $[t_{N-1},t_N]$ and $\tilde{Y}_{t_{N-1}} = u(t_{N-1},X_{t_{N-1}}^*)$. \par 
		Repeat the procedure finite times, we obtain that $Y^* = \tilde{Y}$ on the whole $[0,T]$. Let $Z^* = \tilde{Z}$ and $K^* = \tilde{K}$, it is obvious that $(Y^*,Z^*,K^*)$ is the solution to GBSDE (\ref{bgsde}).	Combining (\ref{gsde}) and (\ref{bgsde}), we know that $(X^*,Y^*,Z^*,K^*)$ is a solution to FBGSDE (\ref{eq}). \par 
		Step 4. Finally, we prove the uniqueness of this solution. In fact, if $(X',Y',Z',K')$ is a solution to the FBGSDE (\ref{eq}), then by the arguments from Step 1 to Step 3, we know that $X'$ and $Y'$ must satisfy $Y'_t = u(t,X_t')$. Thus $X'$ solves the GSDE 
		\begin{equation*}
			\left\{
			\begin{array}{l}
				dX_t = b(t,X_t,u(t,X_t))dt + h(t,X_t,u(t,X_t))d \langle B \rangle_t + \sigma(t,X_t,u(t,X_t))dB_t, \\
				X_0 = x, 
			\end{array}
			\right. 
		\end{equation*}
		which implies that $X' = X^*$ by the arguments in Proposition \ref{smoothexist}. Furthermore, $(Y',Z',K')$ solves the BGSDE 
		\begin{equation*}
			\left\{
			\begin{array}{l}
				dY_t = -f(t,X^*,Y_t,Z_t)dt -g(t,X^*,Y_t,Z_t)d \langle B \rangle_t + Z_tdB_t + dK_t, \\
				Y_T = \Phi(X_T^*), 
			\end{array}
			\right.
		\end{equation*}
		which implies that $(Y',Z',K') = (Y^*,Z^*,K^*)$, by the unique solvability of BGSDEs (\cite{bsde}). Therefore, the uniqueness holds and we complete our proof. 
	\end{proof}
	
	\begin{remark}
		Theorem \ref{thm2} also implies that PDE (\ref{pde*}) admits a unique solution. In fact, if there are two solutions $u$ and $v$, following the argument in Proposition \ref{smoothexist}, $(X^u,Y^u,Z^u,K^u)$ and $(X^v,Y^v,Z^v,K^v)$ are both solutions to FBGSDE (\ref{eq}). Then by Theorem \ref{thm2}, the two solutions are the same, hence, we obtain $u(0,x) = Y_0^u = Y_0^v = v(0,x)$. Since $(0,x)$ can be replaced by any $(t,x) \in [0,T] \times \mathbb R$, by which we have $u \equiv v$, the uniqueness of the solution to PDE (\ref{pde*}) holds true. 
	\end{remark}
	
	Since the conclusions of this paper are inspired mainly by the idea of \cite{fbsdes}, and heavily depends on the results of fully nonlinear PDEs, there are still a lot to improve in this paper, we list some of them here. 
	
	\begin{itemize}
		\item We only consider the case that $b,h$ are independent of $z$ in this paper, for the situation of $b = b(t,x,y,z)$ and $h = h(t,x,y,z)$, do Theorem \ref{thm1} and Theorem \ref{thm2} hold true? 
		\item We require all the coefficients are continuous differentiable with their derivatives satisfying Lipschitz conditions, can this assumption be relaxed to the Lipschitz conditions of coefficients? 
		\item What if $\sigma$ does not satisfy uniformly elliptic condition, or can we get a similar conclusion without using those results of PDEs? 
	\end{itemize}
	
	In fact, the second problem is influenced by the results of PDEs we quote in this paper. If the $C^1$ norm of $u$, which is the solution to PDE (\ref{pde*}), is bounded by a constant only depending on the boundary of first derivatives of coefficients, which means it is independent of the second derivatives of coefficients, then Assumption \ref{h2} can be weaker as we wish. However, the first problem is much more complex, since we need a higher moment control in the Doob's inequality (Proposition \ref{doob}) under G-expectation. As a result, it is not easy to construct a contractive map $\mathcal{I}$ as before, as long as $b$ or $h$ depends on $z$. Actually, if $b$ or $h$ depends on $z$, Lemma \ref{solution} may not hold any longer, because the function $u_x \cdot b(t,x,u,u_x\cdot\sigma)$ or $u_x \cdot h(t,x,u,u_x\cdot\sigma)$ does not satisfy the inequality appearing in (\romannumeral 4) of Definition \ref{def-f}, and we cannot apply Theorem \ref{pthm} to prove Lemma \ref{solution}. Consequently, we cannot use the results of PDEs to construct solutions of a family of FBGSDEs to approach the solution of FBGSDE (\ref{eq}). 
	
	\bibliographystyle{unsrt}
	\bibliography{ref}
	
\end{document}